\documentclass{amsproc}%
\usepackage{amsfonts}
\usepackage{amsmath}
\usepackage{amssymb}
\usepackage{graphicx}%
\providecommand{\U}[1]{\protect\rule{.1in}{.1in}}
\theoremstyle{plain}

\newtheorem{corollary}{Corollary}

\newtheorem{definition}{Definition}

\newtheorem{remark}{Remark}

\newtheorem{theorem}{Theorem}
\numberwithin{equation}{section}
\begin{document}
\title[Similarity Degree]{The Similarity Degree of Some C*-algebras}
\author{Don Hadwin}
\address{University of Mew Hampshire}
\email{don@unh.edu}
\author{Weihua Li}
\curraddr{Columbia College Chicago}
\email{wli@colum.edu}
\subjclass[2000]{Primary 46L05; Secondary 46A22, 46H25, 46M10, 47A20.}
\keywords{Kadison Similarity Problem, similarity degree, tracially nuclear, C*-algebra}

\begin{abstract}
We define the class of weakly approximately divisible unital C*-algebras and
show that this class is closed under direct sums, direct limits, any tensor
product with any C*-algebra, and quotients. A nuclear C*-algebra is weakly
approximately divisible if and only if it has no finite-dimensional
representations. We also show that Pisier's similarity degree of a weakly
approximately divisible C*-algebra is at most 5.

\end{abstract}
\maketitle

\section{Introduction\bigskip}

One of the most famous and oldest open problems in the theory of C*-algebras
is Kadison's Similarity problem \cite{Kadison1}, which asks whether every
bounded unital homomorphism $\rho$ from a C*-algebra $\mathcal{A}$ into the
algebra $B\left(  H\right)  $ of operators on a Hilbert space $H$ must be
similar to a $\ast$-homomorphism, i.e., does there exist an invertible $S\in
B\left(  H\right)  $ such that $\pi\left(  A\right)  =S\rho\left(  A\right)
S^{-1}$ defines a $\ast$-homomorphism? One measure of the quality of a good
problem is the number of interesting equivalent formulations. In this regard
Kadison's problem gets high marks:

\begin{enumerate}
\item[I.] Inner Derivation Problem: (\cite{Kir}, \cite{Christensen1}) If
$\mathcal{M}\subseteq B\left(  H\right)  $ is a von Neumann algebra and
$\delta:\mathcal{M}\rightarrow B\left(  H\right)  $ is a derivation, does
there exist a $T\in B\left(  H\right)  $ such that, for every $A\in
\mathcal{M}$,%
\[
\delta\left(  A\right)  =AT-TA?
\]

\item[II.] Hyperreflexivity Problem: (\cite{Kir}, \cite{Christensen1}) If
$\mathcal{M}\subseteq B\left(  H\right)  $ is a von Neumann algebra, does
there exist a $K$, $1\leq K<\infty$ such that, for every $T\in B\left(
H\right)  ,$%
\[
dist\left(  T,\mathcal{M}\right)  \leq K\sup\left\{  \left\Vert
PT-TP\right\Vert :P\in\mathcal{M}^{\prime}\text{, }P=P^{\ast}=P^{2}\right\}
.
\]

\item[III.] Dixmier's Invariant Operator Range Problem \cite{Dix} (Foia\c{s}
\cite{F}, Pisier \cite[Theorem 10.5]{Pisier 5}, see also \cite{HP}): If
$\mathcal{M}\subseteq B\left(  H\right)  $ is a von Neumann algebra, $A\in
B\left(  H\right)  $ and $T\left(  A\left(  H\right)  \right)  \subseteq
A\left(  H\right)  $ for every $T\in\mathcal{M}$, then does there exist
$D\in\mathcal{M}^{\prime}$ such that $A\left(  H\right)  =D\left(  H\right)
$? Vern Paulsen \cite{P2} proved that an affirmative answer is equivalent to
the assertion that the range of $A\oplus A\oplus\cdots$ is invariant for
$\mathcal{M}\otimes\mathcal{K}\left(  \ell^{2}\right)  $.
\end{enumerate}

In \cite{Haagerup 1} U. Haagerup proved that Kadison's question has an
affirmative answer whenever the representation $\rho$ has a cyclic vector, a
result that is independent of the structure of the algebra $\mathcal{A}$.
Haagerup \cite{Haagerup 1} also showed that a homomorphism $\rho$ is similar
to a $\ast$-homomorphism if and only if $\rho$ is completely bounded (see also
\cite{Ch2}; see the union of \cite{H} and \cite{W} for another proof; see
\cite{Paulsen} and \cite{P2} for a lovely exposition of these ideas). In
\cite{Pisier 1} G. Pisier proved that, for a fixed C*-algebra $\mathcal{A}$,
every bounded homomorphism of $\mathcal{A}$ is similar to a $\ast
$-homomorphism if and only if $\mathcal{A}$ satisfies a certain factorization
property. It was shown in \cite{HP} that Kadison's similarity property is
universally true if and only if there is a Pisier-like factorization in terms
of scalar matrices and noncommutative polynomials that is independent of the
C*-algebra. It was also shown in \cite{HP} that if $H=\ell^{2}\oplus\ell
^{2}\oplus\cdots$ and $D=1\oplus\frac{1}{2}\oplus\frac{1}{2^{2}}\oplus\cdots$
and $\mathcal{S}$ is the unital algebra of all operators $T\in B\left(
H\right)  $ with an operator matrix $T=\left(  A_{ij}\right)  $ such that
$\rho\left(  T\right)  =D^{-1}TD=\left(  2^{j-i}A_{ij}\right)  $ is bounded,
then Kadison's similarity problem has an affirmative answer if and only if,
for every unital C*-subalgebra $\mathcal{A}$ of $\mathcal{S}$, the
homomorphism $\rho|_{\mathcal{A}}$ is similar to a $\ast$-homomorphism.

Our main focus in this paper is another amazing result of G. Pisier
\cite{Pisier 1} where he shows that, for a unital C*-algebra $\mathcal{A}$,
Kadison's similarity property holds for $\mathcal{A}$ if and only if there is
a positive number $d$ for which there is a positive number $K$ such that%
\[
\left\Vert \rho\right\Vert _{cb}\leq K\left\Vert \rho\right\Vert ^{d}%
\]
for every bounded unital homomorphism $\rho$ on $\mathcal{A}$. Pisier proved
that the smallest such $d$ is an integer which he calls the \emph{similarity
degree} $d\left(  \mathcal{A}\right)  $ of $\mathcal{A}$. Here are a few
results on the similarity degree:

\begin{enumerate}
\item $\mathcal{A}$ is nuclear if and only if $d({\mathcal{A}})=2$
(\cite{Bunce1}, \cite{Christensen1}, \cite{Pisier2});

\item if ${\mathcal{A}}={\mathcal{B}(\mathcal{H})}$, then $d({\mathcal{A}})=3$
(\cite{Pisier3});

\item $d({\mathcal{A}}\otimes{\mathcal{K}(\mathcal{H})})\leq3$ for any
C$^{\ast}$-algebra $\mathcal{A}$ (\cite{Haagerup 1}, \cite{Pisier4});

\item if $\mathcal{M}$ is a factor of type $II_{1}$ with property $\Gamma$,
then $d({\mathcal{M}})=3$ (\cite{Christensen2});

\item if $\mathcal{A}$ is an approximately divisible C*-algebra, then
$d\left(  \mathcal{A}\right)  \leq5$ \cite{li}.

\item if $\mathcal{A}$ is nuclear and contains unital matrix algebras of any
order, then $d(\mathcal{A}\otimes\mathcal{B})\leq5$ for any unital C*-algebra
$\mathcal{B}$ (\cite{F. Pop});

\item if $\mathcal{A}$ is nuclear and contains finite-dimensional
C*-subalgebras of arbitrarily large subrank (see the definition below), then
$d(\mathcal{A}\otimes\mathcal{B})\leq5$ for any unital C*-algebra
$\mathcal{B}$ \cite{li};

\item if $\mathcal{A}$ is nuclear and contains homomorphic images of certain
dimension-drop C*-algebras $\mathcal{Z}_{p,q}$ for all relatively prime
integers $p,q$ (e.g., $\mathcal{A}$ contains a copy of the JiangSu algebra),
then $d(\mathcal{A}\otimes\mathcal{B})\leq5$ for any unital C*-algebra
$\mathcal{B}$ \cite{JW}.
\end{enumerate}

\bigskip

In this paper we define the class of weakly approximately divisible
C*-algebras and show that this class is closed under unital $\ast
$-homomorphisms, arbitrary tensor products and direct limits. We also define
the class of tracially nuclear C*-algebras that properly contains the class of
nuclear C*-algebras, and we show that a tracially nuclear C*-algebra is weakly
approximately divisible if and only if it has no finite-dimensional
representations. We prove that if $\mathcal{A}$ is weakly approximately
divisible, then $d\left(  \mathcal{A}\right)  \leq5$. We extend the results
$6$-$8$ above to the case when $\mathcal{A}$ is tracially nuclear and has no
finite-dimensional representations, and the tensor product is with respect to
any C*-crossnorm.

\bigskip

\section{Weakly approximately divisible algebras}

\bigskip

\bigskip

If $\tau$ is a tracial state on $\mathcal{M}$, we let $\left\Vert
\cdot\right\Vert _{\tau}$ denote the seminorm on $\mathcal{M}$ defined in the
GNS construction by%
\[
\left\Vert a\right\Vert _{\tau}^{2}=\tau\left(  a^{\ast}a\right)  .
\]

Suppose $\mathcal{B}$ is a finite-dimensional unital C*-subalgebra of a unital
C*-algebra $\mathcal{A}$. First we know that $\mathcal{B}$ is $\ast
$-isomorphic to $\mathcal{M}_{k_{1}}\left(  \mathbb{C}\right)  \oplus
\cdots\oplus\mathcal{M}_{k_{m}}\left(  \mathbb{C}\right)  $ and the
\emph{subrank}$\left(  \mathcal{B}\right)  $ is defined to be $\min\left(
k_{1},\ldots,k_{m}\right)  .$ Note that if $\pi:\mathcal{B}\rightarrow
\mathcal{D}$ is a unital $\ast$-homomorphism, then
\[
\emph{subrank}\left(  \mathcal{B}\right)  \leq\emph{subrank}\left(  \pi\left(
\mathcal{B}\right)  \right)  .
\]
If $P_{1}=1\oplus0\oplus\cdots\oplus0,$ $P_{2}=0\oplus1\oplus0\oplus
\cdots\oplus0,\ldots,P_{m}=0\oplus\cdots\oplus1$ are the minimal central
projections of $\mathcal{B}$, then, for $1\leq s\leq m$, we have
$P_{s}\mathcal{A}P_{s}$ is isomorphic to $\mathcal{M}_{k_{s}}\left(
\mathbb{C}\right)  \otimes\mathcal{A}_{s}=\mathcal{M}_{k_{s}}\left(
\mathcal{A}_{s}\right)  $ for some algebra $\mathcal{A}_{s}$. The relative
commutant of $\mathcal{M}_{k_{s}}\left(  \mathbb{C}\right)  $ in
$\mathcal{M}_{k_{s}}\left(  \mathcal{A}_{s}\right)  $ is
\[
\mathcal{D}_{s}=\left\{  \left(
\begin{array}
[c]{cccc}%
A &  &  & \\
& A &  & \\
&  & \ddots & \\
&  &  & A
\end{array}
\right)  :A\in\mathcal{A}_{s}\right\}  ,
\]
and the relative commutant of $\mathcal{B}$ in $\mathcal{A}$ is $\mathcal{D}%
_{1}\oplus\cdots\oplus\mathcal{D}_{m}$. Suppose $T\in\mathcal{A}$, and
$P_{s}TP_{s}=\left(  a_{ijs}\right)  _{1\leq i,j\leq k_{s}}$. Let
$D_{s}=diag\left(  c,\ldots,c\right)  $ where $c=\frac{1}{k_{s}}\left(
a_{11s}+\cdots+a_{k_{s}k_{s}s}\right)  $. The map $E_{\mathcal{B}}%
:\mathcal{A}\rightarrow\mathcal{B}^{\prime}\cap\mathcal{A}$ sending $T$ to
$D_{1}\oplus\cdots\oplus D_{m}$ is called the conditional expectation from
$\mathcal{A}$ to $\mathcal{B}^{\prime}\cap\mathcal{A}$ and is a completely
positive unital idempotent. For $1\leq s\leq m$, let $\mathcal{G}_{s}$ be the
group of all matrices in $\mathcal{M}_{k_{s}}\left(  \mathbb{C}\right)  $ suct
that the only one nonzero entry in each row and each column is $1$ or $-1$,
and let $\mathcal{G}=\mathcal{G}_{1}\oplus\cdots\oplus\mathcal{G}_{m}%
\subseteq\mathcal{B}$. Then we have%
\begin{equation}
E_{\mathcal{B}}\left(  T\right)  =\frac{1}{Card\mathcal{G}}%
{\displaystyle\sum_{U\in\mathcal{G}}}
UTU^{\ast}. \tag{$\left(     \ast\right)     $}%
\end{equation}
Moreover, if $S\in\mathcal{B}^{\prime}\cap\mathcal{A}$ and $T\in\mathcal{A}$,
then%
\[
E_{\mathcal{B}}\left(  ST\right)  =SE_{\mathcal{B}}\left(  T\right)  \text{
and }E_{\mathcal{B}}\left(  TS\right)  =E_{\mathcal{B}}\left(  T\right)
S\text{.}%
\]
Furthermore, if $\tau$ is a tracial state on $\mathcal{A}$, then, for every
$A\in\mathcal{A}$,
\[
\left\Vert E_{\mathcal{B}}\left(  A\right)  \right\Vert _{\tau}\leq\left\Vert
A\right\Vert _{\tau}.
\]

Suppose $\mathcal{M}$ is a von Neumann algebra and $\left\{  v_{i}:i\in
I\right\}  \subseteq\mathcal{M}$ is a family satisfying $\sum_{i\in I}%
v_{i}^{\ast}v_{i}=1$ (convergence is in the weak*-topology). Then
$\varphi\left(  T\right)  =\sum_{i\in I}v_{i}^{\ast}Tv_{i}$ defines a unital
completely positive map from $\mathcal{M}$ to $\mathcal{M}$. Let us call such
a map \emph{internally spatial}, and call a unital completely positive map
\emph{internal} if it is a convex combination of internally spatial maps
on\emph{ }$\mathcal{M}$.

\begin{remark}
\label{int}There are two key properties of internal maps:

\begin{enumerate}
\item They can be pushed forward through normal unital $\ast$-homomorphisms
between von Neumann algebras. Suppose $\mathcal{M}$ and $\mathcal{N}$ are von
Neumann algebras and $\rho:\mathcal{M}\rightarrow\mathcal{N}$ is a unital
weak*-weak*-continuous unital $\ast$-homomorphism, and suppose $\left\{
v_{i}:i\in I\right\}  \subseteq\mathcal{M}$ with $\sum_{i\in I}v_{i}^{\ast
}v_{i}=1$ and $\varphi\left(  T\right)  =\sum_{i\in I}v_{i}^{\ast}Tv_{i}$.
Then $\left\{  \pi\left(  v_{i}\right)  :i\in I\right\}  \subseteq\mathcal{N}$
and
\[
1=\pi\left(  1\right)  =\pi\left(  \sum_{i\in I}v_{i}^{\ast}v_{i}\right)
=\sum_{i\in I}\pi\left(  v_{i}\right)  ^{\ast}\pi\left(  v_{i}\right)  .
\]
We define $\varphi^{\pi}\left(  S\right)  =\sum_{i\in I}\pi\left(
v_{i}\right)  ^{\ast}S\pi\left(  v_{i}\right)  $, and we have, for every
$a\in\mathcal{M}$%
\[
\varphi^{\pi}\left(  \pi\left(  a\right)  \right)  =\pi\left(  \varphi\left(
a\right)  \right)  .
\]
So if $b\in\pi\left(  \mathcal{A}\right)  $ and $b=\pi\left(  a\right)  ,$
then $\varphi^{\pi}\left(  b\right)  =\pi\left(  \varphi\left(  a\right)
\right)  ,$ which independent of $a$. For a general $\varphi$ this only makes
sense when $\varphi\left(  \ker\pi\right)  \subseteq\ker\pi$. It follows that
$\varphi^{\pi}$ makes sense when $\varphi$ is an internal map, and in this
case, $\varphi^{\pi}$ is an internal map on $\mathcal{N}$.

\item If $\varphi\left(  T\right)  =\sum_{i\in I}v_{i}^{\ast}Tv_{i}$ and $T$
commutes with each $v_{i}$, then, for every $S,$ we have%
\[
\varphi\left(  ST\right)  =\varphi\left(  S\right)  T.
\]
Hence if $\psi$ is a convex combination of spatially internal maps defined in
terms of elements commuting with an operator $T,$ we have $\psi\left(
ST\right)  =\psi\left(  S\right)  T$.
\end{enumerate}
\end{remark}

\begin{definition}
\label{WAD}We say that a unital C*-algebra $\mathcal{A}$ is \emph{weakly
approximately divisible} if and only if, for every finite subset $\mathcal{F}$
of $\mathcal{A}$ there is a net $\left\{  \left(  \mathcal{B}_{\lambda
},\varphi_{\lambda}\right)  \right\}  _{\lambda\in\Lambda}$ where each
$\mathcal{B}_{\lambda}$ is a finite-dimensional unital C*-subalgebra of
$\mathcal{A}^{\#\#}$ and $\varphi_{\lambda}$ is an internal completely
positive map such that

\begin{enumerate}
\item $\lim_{\lambda}$\textrm{subrank}$\left(  \mathcal{B}_{\lambda}\right)
=\infty,$

\item $\varphi_{\lambda}:\mathcal{A}\rightarrow\mathcal{B}_{\lambda}^{\prime
}\cap\mathcal{A}^{\#\#},$

\item For every $a\in\mathcal{F}$, $\varphi_{\lambda}\left(  a\right)
\rightarrow a$ in the weak*-topology on $\mathcal{A}^{\#\#}$.
\end{enumerate}
\end{definition}

\begin{remark}
Suppose $n$ is a positive integer and let $\mathcal{V}_{n}$ be the set of
$n$-tuples $\left(  a_{1},\ldots,a_{n}\right)  $ of elements in $\mathcal{A}$
such that the conditions in Definition \ref{WAD} hold when $\mathcal{F}%
\mathbf{=}\left\{  a_{1},\ldots,a_{n}\right\}  $. Suppose $U_{k}$ is a
weak*-neighborhood of $a_{k}$ in $\mathcal{A}^{\#\#}$ for $1\leq k\leq n.$
Since addition on $\mathcal{A}^{\#\#}$ is weak*-continuous, there is a
weak*-neighborhood $V_{k}$ of $a_{k}$ and a weak*-neighborhood $E$ of $0$ such
that%
\[
V_{k}+E\subseteq U_{k}%
\]
for $1\leq k\leq n$. Suppose $\left(  b_{1},\ldots,b_{n}\right)  $ is in the
norm closure of $\mathcal{V}_{n}$ and suppose $U_{k}$ is a weak*-neighborhood
of $b_{k}$ in $\mathcal{A}^{\#\#}$ for $1\leq k\leq n.$ Since addition on
$\mathcal{A}^{\#\#}$ is weak*-continuous, there is a weak*-neighborhood
$V_{k}$ of $b_{k}$ and a weak*-neighborhood $E$ of $0$ such that%
\[
V_{k}+E\subseteq U_{k}%
\]
for $1\leq k\leq n$. Since $0\in E$ and $E$ is weak*-open, there is an
$\varepsilon>0$ such that $\left\{  x\in\mathcal{A}^{\#\#}:\left\Vert
x\right\Vert <\varepsilon\right\}  \subseteq E.$ Now choose $\left(
a_{1},\ldots,a_{n}\right)  \in\mathcal{V}_{n}$ so that $a_{k}\in V_{k}$ and
$\left\Vert a_{k}-b_{k}\right\Vert <\varepsilon$ for $1\leq k\leq n$. Next
suppose $m$ is a positive integer. It follows from the definition of
$\mathcal{V}_{n}$ that there is a finite-dimensional C*-subalgebra
$\mathcal{B}$ of $\mathcal{A}^{\#\#}$ and a completely positive unital map
$\varphi:\mathcal{A\rightarrow B}^{\prime}\cap\mathcal{A}^{\#\#}$ such that
\textrm{subrank}$\left(  \mathcal{B}\right)  \geq m$ and such that
$\varphi\left(  a_{k}\right)  \in V_{k}$ for $1\leq k\leq n$. It follows that
$\varphi\left(  b_{k}\right)  -\varphi\left(  a_{k}\right)  =\varphi\left(
b_{k}-a_{k}\right)  \in E$ for $1\leq k\leq n$, so
\[
\varphi\left(  b_{k}\right)  \in V_{k}+E\subseteq U_{k}%
\]
for $1\leq k\leq n$. Hence $\left(  b_{1},\ldots,b_{n}\right)  \in
\mathcal{V}_{n}$. Thus $\mathcal{V}_{n}$ is norm closed. It is also clear that
$\mathcal{V}_{n}$ is a linear space. Hence, to verify that $\mathcal{A}$ is
weakly approximately divisible, it is sufficient to show that the conditions
of Definition \ref{WAD} holds for all finite subsets $\mathcal{F}$ of a set
$W$ whose norm closed linear span $\overline{sp}\left(  W\right)  $ is
$\mathcal{A}$.
\end{remark}

Recall \cite{Tak 2} that a C*-algebra $\mathcal{A}$ is \emph{nuclear} if, for
every Hilbert space $H$ and every unital $\ast$-homomorphism $\pi
:\mathcal{A}\rightarrow B\left(  H\right)  $ we have $\pi\left(
\mathcal{A}\right)  ^{\prime\prime}$ is a hyperfinite von Neumann algebra. We
say that $\mathcal{A}$ is \emph{tracially nuclear} if, for every tracial state
$\tau$ on $\mathcal{A}$ with GNS representation $\pi_{\tau}$ we have
$\pi_{\tau}\left(  \mathcal{A}\right)  ^{\prime\prime}$ is a hyperfinite von
Neumann algebra. As a flip side of the notion of RFD C*-algebras, we say that
a unital C*-algebra $\mathcal{A}$ is \emph{NFD} if $\mathcal{A}$ has no unital
finite-dimensional representations.

\begin{theorem}
Suppose $\mathcal{A}$ and $\mathcal{D}$ are unital C*-algebras. Then

\begin{enumerate}
\item If $\mathcal{A}$ is approximately divisible, then $\mathcal{A}$ is
weakly approximately divisible.

\item If $\mathcal{A}$ is weakly approximately divisible and $\pi
:\mathcal{A}\rightarrow\mathcal{D}$ is a surjective unital $\ast
$-homomorphism, then $\mathcal{D}$ is weakly approximately divisible.

\item If $\mathcal{A}$ is weakly approximately divisible, then $\mathcal{A}$
has no finite-dimensional representations.

\item If $\mathcal{A}$ is weakly approximately divisible, then $\mathcal{A}%
\otimes_{\mathrm{\max}}\mathcal{D}$ is weakly approximately divisible.

\item A finite direct sum $\sum_{1\leq k\leq n}^{\oplus}\mathcal{A}_{k}$ of
unital C*-algebras is weakly approximately divisible if and only if each
summand $\mathcal{A}_{k}$ is weakly approximately divisible.

\item If $n$ is a positive integer, then $\mathcal{A}\otimes\mathcal{M}%
_{n}\left(  \mathbb{C}\right)  $ is weakly approximately divisible if and only
if $\mathcal{A}$ is.

\item A direct limit of weakly approximately divisible C*-algebras is weakly
approximately divisible.

\item If $\mathcal{A}$ is an NFD C*-algebra and $\mathcal{M}$ is the type
$II_{1}$ direct summand of $\mathcal{A}^{\#\#}$ and $\gamma:\mathcal{A}%
\rightarrow\mathcal{M}$ is the inclusion into $\mathcal{A}^{\#\#}$ followed by
the projection map, then $\mathcal{A}$ is weakly approximately divisible if
and only if, for every finite subset $\mathcal{F}\subseteq\mathcal{A}$ there
is a net $\left\{  \left(  \mathcal{B}_{\lambda},\varphi_{\lambda}\right)
\right\}  $ where $\mathcal{B}_{\lambda}$ is a finite-dimensional
C*-subalgebra of $\mathcal{M}$, $\varphi_{\lambda}$ is an internal map on
$\mathcal{M}$ and%
\[
\varphi_{\lambda}\left(  \pi\left(  a\right)  \right)  \rightarrow
\gamma\left(  a\right)
\]
in the weak*-topology for every $a\in\mathcal{F}$.

\item If $\mathcal{A}$ is tracially nuclear, then $\mathcal{A}$ is weakly
approximately divisible if and only if $\mathcal{A}$ is NFD.

\item If $\mathcal{A}$ is nuclear, then $\mathcal{A}$ is weakly approximately
divisible if and only if $\mathcal{A}$ is NFD.
\end{enumerate}
\end{theorem}

\begin{proof}
$\left(  1\right)  .$ This follows immediately from the definitions.

$\left(  2\right)  .$ If $\pi:\mathcal{A}\rightarrow\mathcal{D}$ is a
surjective unital $\ast$-homomorphism, then $\pi$ extends to a
weak*-weak*-continuous surjective unital $\ast$-homomorphism $\rho
:\mathcal{A}^{\#\#}\rightarrow\mathcal{D}^{\#\#}$. Given $d_{1},\ldots
,d_{n}\in\mathcal{D}$, choose $a_{1},\ldots,a_{n}\in\mathcal{A}$ so that
$\pi\left(  a_{k}\right)  =d_{k}$ for $1\leq k\leq n.$ Choose a net $\left\{
\left(  \mathcal{B}_{\lambda},\varphi_{\lambda}\right)  \right\}  $ according
to Definition \ref{WAD} with $\mathcal{F}=\left\{  a_{1},\ldots,a_{n}\right\}
$. It follows that $\varphi_{\lambda}^{\rho}$ is an internal completely
positive map on $\mathcal{D}^{\#\#}$ and
\[
\varphi_{\lambda}^{\rho}\left(  \mathcal{D}\right)  =\varphi_{\lambda}^{\rho
}\left(  \rho\left(  \mathcal{A}\right)  \right)  =\rho\left(  \varphi
_{\lambda}\left(  \mathcal{A}\right)  \right)  \subseteq
\]%
\[
\rho\left(  \mathcal{B}_{\lambda}^{\prime}\cap\mathcal{A}^{\#\#}\right)
\subseteq\rho\left(  \mathcal{B}_{\lambda}\right)  ^{\prime}\cap
\mathcal{D}^{\#\#}.
\]
Further for each $d_{k}$ we have%
\[
\text{w*-}\lim_{\lambda}\varphi_{\lambda}^{\rho}\left(  d_{k}\right)
=\text{w*-}\lim_{\lambda}\rho\left(  \varphi_{\lambda}\left(  a_{k}\right)
\right)  =\rho\left(  a_{k}\right)  =d_{k},
\]
since $\rho$ is weak*-weak*-continuous. Since \textrm{subrank}$\left(
\mathcal{B}_{\lambda}\right)  \leq$\textrm{subrank}$\left(  \rho\left(
\mathcal{B}_{\lambda}\right)  \right)  $, we conclude that $\mathcal{D}$ is
weakly approximately divisible.

$\left(  3\right)  .$ This follows from $\left(  2\right)  $ and the obvious
fact that no finite-dimensional C*-algebra is weakly approximately divisible.

$\left(  4\right)  .$ Let $\rho:\mathcal{A}\otimes_{\mathrm{\max}}%
\mathcal{D}\rightarrow\left(  \mathcal{A}\otimes_{\mathrm{\max}}%
\mathcal{D}\right)  ^{\#\#}$ be the natural inclusion map. We can assume
$\left(  \mathcal{A}\otimes_{\mathrm{\max}}\mathcal{D}\right)  ^{\#\#}%
\subseteq B\left(  H\right)  $ for some Hilbert space $H$ so that, on bounded
subsets of $\left(  \mathcal{A}\otimes_{\mathrm{\max}}\mathcal{D}\right)
^{\#\#}$, the weak*-topology coincides with the weak-operator topology. If
$\rho:\mathcal{A\rightarrow A}\otimes1\mathcal{\subseteq}\mathcal{A}%
\otimes_{\mathrm{\max}}\mathcal{D}$ is the inclusion map, then there is a
weak*-weak*-continuous unital $\ast$-homomorphism $\sigma:\mathcal{A}%
^{\#\#}\rightarrow\left(  \mathcal{A}\otimes_{\mathrm{\max}}\mathcal{D}%
\right)  ^{\#\#}$ such that the restriction of $\sigma$ to $\mathcal{A}$ is
$\rho$. Let $W=\left\{  a\otimes b:a\in\mathcal{A},b\in\mathcal{B}\right\}  $.
Clearly, $\overline{sp}W=\mathcal{A}\otimes_{\max}\mathcal{B}$ (where the
closure is with respect to $\left\Vert {}\right\Vert _{\max}$). Suppose
$a_{1}\otimes b_{1},\ldots,a_{n}\otimes b_{n}\in W$. Since $\mathcal{A}$ is
weakly approximately divisible, we can choose a net $\left\{  \left(
\mathcal{B}_{\lambda},\varphi_{\lambda}\right)  \right\}  $ as in Definition
\ref{WAD}. We know that $\left\{  \varphi_{\lambda}^{\sigma}\right\}  $ is a
net of internal maps on $\left(  \mathcal{A}\otimes_{\mathrm{\max}}%
\mathcal{D}\right)  ^{\#\#}$ and
\[
\varphi_{\lambda}^{\sigma}\left(  a_{k}\otimes1\right)  =\varphi_{\lambda
}^{\sigma}\left(  \sigma\left(  a_{k}\right)  \right)  =\sigma\left(
\varphi_{\lambda}\left(  a_{k}\right)  \right)  \rightarrow\sigma\left(
a_{k}\right)  =a_{k}\otimes1
\]
in the weak*-topology for $1\leq k\leq n$. On the other hand each
$\varphi_{\lambda}$ is a convex combination of spatially internal maps defined
by partial isometries in $\mathcal{A}^{\#\#},$ so each $\varphi_{\lambda
}^{\sigma}$ is a convex combination of spatially internal maps defined by
partial isometries in $\sigma\left(  \mathcal{A}^{\#\#}\right)  $ which is
contained in $\left(  \mathcal{A}\otimes_{\mathrm{\max}}\mathcal{D}\right)
^{\#\#}\cap\left(  1\otimes\mathcal{D}\right)  ^{^{\prime}}$. Hence, for every
$S\in\left(  \mathcal{A}\otimes_{\mathrm{\max}}\mathcal{D}\right)  ^{\#\#}$
and every $d\in\mathcal{D}$, we have
\[
\varphi_{\lambda}^{\sigma}\left(  S\left(  1\otimes d\right)  \right)
=\varphi_{\lambda}^{\sigma}\left(  S\right)  \left(  1\otimes d\right)  .
\]
Hence, for $1\leq k\leq n$%
\[
\varphi_{\lambda}^{\sigma}\left(  a_{k}\otimes d_{k}\right)  =\varphi
_{\lambda}^{\sigma}\left(  \left(  a_{k}\otimes1\right)  \left(  1\otimes
d_{k}\right)  \right)  =\varphi_{\lambda}^{\sigma}\left(  a_{k}\otimes
1\right)  \left(  1\otimes d_{k}\right)  .
\]
But $\varphi_{\lambda}^{\sigma}\left(  a_{k}\otimes1\right)  \rightarrow
a_{k}\otimes1$ in the weak*-topology. Hence%
\[
\varphi_{\lambda}^{\sigma}\left(  a_{k}\otimes d_{k}\right)  \rightarrow
a_{k}\otimes d_{k}%
\]
in the weak*-topology on $\left(  \mathcal{A}\otimes_{\max}\mathcal{B}\right)
^{\#\#}$ for $1\leq k\leq n$. Since, for every $\lambda$,
\[
\mathrm{subrank}\left(  \mathcal{B}_{\lambda}\right)  \leq\mathrm{subrank}%
\left(  \sigma\left(  \mathcal{B}_{\lambda}\right)  \right)  ,
\]
we see that $\mathcal{A}\otimes_{\max}\mathcal{B}$ is weakly approximately divisible.

$\left(  5\right)  .$ This easily follows from the fact that $\left(
\sum_{1\leq k\leq n}^{\oplus}\mathcal{A}_{k}\right)  ^{\#\#}=\sum_{1\leq k\leq
n}^{\oplus}\mathcal{A}_{k}^{\#\#}$.

$\left(  6\right)  .$ This is clear, since $\left(  \mathcal{A}\otimes
\mathcal{M}_{n}\left(  \mathbb{C}\right)  \right)  ^{\#\#}$ is isomorphic to
$\mathcal{A}^{\#\#}\otimes\mathcal{M}_{n}\left(  \mathbb{C}\right)  $.

$\left(  7\right)  .$ Suppose $\left\{  \mathcal{A}_{i}:i\in I\right\}  $ is
an increasingly directed family of C*-subalgebras of $\mathcal{A}$ such that
$W=\cup_{i\in I}\mathcal{A}_{i}$ is dense in $\mathcal{A}$. Suppose
$\mathcal{F}\subseteq W$ is finite. Then there is an $i\in I$ such that
$\mathcal{F}\subseteq\mathcal{A}_{i}$. If $\rho:\mathcal{A}_{i}\rightarrow
\mathcal{A}$ is the inclusion map, there is a unital weak*-weak*-continuous
unital $\ast$-homomorphism $\sigma:\mathcal{A}_{i}^{\#\#}\rightarrow
\mathcal{A}^{\#\#}$ whose restriction to $\mathcal{A}_{i}$ is $\rho$. The rest
follows as in the proof of $\left(  2\right)  $.

$\left(  8\right)  .$ If $\mathcal{A}$ is weakly approximately divisible, then
for a finite subset $\mathcal{F}\subseteq\mathcal{A}$ we can find a net
$\left\{  \left(  \mathcal{B}_{\lambda},\varphi_{\lambda}\right)  \right\}  $
as in Definition \ref{WAD} that works in $\mathcal{A}^{\#\#}$, and if we
project all of this onto $\mathcal{M}$, we get the desired net. Now suppose
$\mathcal{A}$ satisfies the condition in $\left(  8\right)  $. We can write
$\mathcal{A}^{\#\#}=\mathcal{M}\oplus\mathcal{N}$, and since $\mathcal{A}$ has
no finite-dimensional representations, $\mathcal{N}$ is the direct sum of a
type $I_{\infty}$ algebra, a $II_{\infty}$ and a type $III$ algebra. In
particular this means that there is an orthogonal sequence $\left\{
P_{n}\right\}  $ of pairwise Murray-von Neumann equivalent projections whose
sum is $1$. Suppose $N$ is a positive integer, and let $Q_{k}=\sum_{j=\left(
k-1\right)  N+1}^{kN}P_{j}$. Then $\left\{  Q_{n}\right\}  $ is an orthogonal
sequence of pairwise equivalent projections whose sum is $1$. We can construct
a system of matrix units $\left\{  E_{ij}\right\}  _{1\leq i,j<\infty}$ so
that $E_{kk}=Q_{k}$ for all $k\geq1.$ Then every $T\in\mathcal{N}$ has an
infinite operator matrix $T=\left(  T_{ij}\right)  .$ The map
\[
\psi_{N}\left(  T\right)  =diag\left(  T_{11},T_{11},\ldots\right)
=\sum_{j=1}^{\infty}E_{j1}TE_{j1}^{\ast}%
\]
is spatially internal and, for every $T$%
\[
\left(  \sum_{k=1}^{N}P_{k}\right)  \psi_{N}\left(  T\right)  \left(
\sum_{k=1}^{N}P_{k}\right)  =\left(  \sum_{k=1}^{N}P_{k}\right)  T\left(
\sum_{k=1}^{N}P_{k}\right)  \rightarrow T
\]
in the weak*-topology. Hence $\psi_{N}\left(  T\right)  \rightarrow T$ in the
weak*-topology. Moreover, $\mathcal{N\cap}\psi_{N}\left(  \mathcal{N}\right)
^{\prime}$ contains full matrix algebras of all orders. Next suppose
$\mathcal{F}$ $\subseteq\mathcal{A}$ is finite. For each $A\in\mathcal{F}$ we
write $A=\gamma\left(  A\right)  \oplus T_{A}$ relative to $\mathcal{A}%
^{\#\#}=\mathcal{M}\oplus\mathcal{N}$. Given the net $\left\{  \left(
\mathcal{B}_{\lambda},\varphi_{\lambda}\right)  \right\}  $ in $\mathcal{M}$
based on our assumption on $\mathcal{A}$, we let $N_{\lambda}=$
\textrm{subrank}$\left(  \mathcal{B}_{\lambda}\right)  $ and choose a full
$N_{\lambda}\times N_{\lambda}$ matrix algebra $\mathcal{C}_{\lambda}$ in
$\mathcal{N\cap}\psi_{N}\left(  \mathcal{N}\right)  ^{\prime}$. Then
$\tau_{\lambda}\left(  S\oplus T\right)  =\varphi_{\lambda}\left(  S\right)
\oplus\psi_{N_{\lambda}}\left(  T\right)  $ is an internal map on
$\mathcal{A}^{\#\#}$ whose range is in $\left(  \mathcal{B}_{\lambda}%
\oplus\mathcal{C}_{\lambda}\right)  ^{\prime}\cap\mathcal{A}^{\#\#}$ such
that
\[
\tau_{\lambda}\left(  A\right)  \rightarrow A
\]
in the weak*-topology for every $A\in\mathcal{F}$. Hence $\mathcal{A}$ is
weakly approximately divisible.

$\left(  9\right)  .$ Let $\mathcal{M}$ and $\gamma$ be as in $\left(
8\right)  $. Let $\Lambda$ be the set of all triples $\lambda=\left(
\mathcal{F}_{\lambda},\mathcal{T}_{\lambda},k_{\lambda}\right)  $ where
$\mathcal{F}_{\lambda}\subseteq\mathcal{A}$ is finite, $\mathcal{T}_{\lambda}$
is a finite set of normal tracial states on $\mathcal{M}$, and $k_{\lambda}%
\in\mathbb{N}$. With the ordering $\left(  \subseteq,\subseteq,\leq\right)  $
we see that $\Lambda$ is a directed set. If $\tau$ is a tracial state on
$\mathcal{M}$, we let $\left\Vert \cdot\right\Vert _{\tau}$ denote the
seminorm on $\mathcal{M}$ defined by%
\[
\left\Vert A\right\Vert _{\tau}=\tau\left(  A^{\ast}A\right)  ^{1/2}.
\]

Suppose $\lambda\in\Lambda$. There is a central projection $P\in\mathcal{M}$
so that $\mathcal{M}=\mathcal{M}_{a}\oplus\mathcal{M}_{s}$ ($\mathcal{M}%
_{a}=P\mathcal{M}$) and so that $\gamma=\gamma_{a}\oplus\gamma_{s}$ and such
that $\gamma_{a}<<%
{\displaystyle\sum_{\tau\in\mathcal{T}_{\lambda}}^{\oplus}}
\pi_{\tau}$ and $\gamma_{s}$ disjoint from $%
{\displaystyle\sum_{\tau\in\mathcal{T}_{\lambda}}^{\oplus}}
\pi_{\tau}$. Since, by assumption, $\left(
{\displaystyle\sum_{\tau\in\mathcal{T}_{\lambda}}^{\oplus}}
\pi_{\tau}\right)  \left(  \mathcal{A}\right)  ^{\prime\prime}=\mathcal{M}%
_{a}$ is hyperfinite. Hence, there is a finite-dimensional unital subalgebra
$\mathcal{D}_{\lambda}$ of $\mathcal{M}_{a}$ and a contractive map
$\eta:\mathcal{F}_{\gamma}\rightarrow\mathcal{D}_{\lambda}$ such that
\[
\max_{\tau\in\mathcal{T}_{\lambda},A\in\mathcal{F}_{\lambda}}\left\Vert
P\gamma\left(  A\right)  -\eta\left(  A\right)  \right\Vert _{\tau}<\frac
{1}{k}.
\]
Note that $\left\Vert T\right\Vert _{\tau}=\left\Vert PT\right\Vert _{\tau}$
for every $T\in\mathcal{M}$ and every $\tau\in\mathcal{T}_{\lambda}$. The
relative commutant $\mathcal{D}_{\lambda}^{\prime}\cap\mathcal{M}_{a}$ is also
a $II_{1}$ von Neumann algebra, so there are $k_{\lambda}$ mutually orthogonal
unitarily equivalent projections in $\mathcal{D}_{\lambda}^{\prime}%
\cap\mathcal{M}_{a}$ whose sum is $1$. Hence $\mathcal{D}_{\lambda}^{\prime
}\cap\mathcal{M}_{a}$ contains a unital subalgebra $\mathcal{E}_{\lambda}$
that is isomorphic to $\mathcal{M}_{k}\left(  \mathbb{C}\right)  $. Similarly,
$\mathcal{M}_{s}$ (if it is not $0$) is a $II_{1}$ von Neumann algebra and
contains an isomorphic copy $\mathcal{G}_{\lambda}$ of $\mathcal{M}%
_{k_{\lambda}}\left(  \mathbb{C}\right)  $. Then $\mathcal{B}_{\lambda
}=\mathcal{E}_{\lambda}\oplus\mathcal{G}_{\lambda}$ is finite-dimensional and
$\emph{subrank}\left(  \mathcal{B}_{\lambda}\right)  =k_{\lambda}$. Define
$\varphi_{\lambda}=E_{\mathcal{B}_{\lambda}}$. For every $A\in\mathcal{F}%
_{\lambda}$ and $\tau\in\mathcal{T}_{\lambda}$, we have%
\[
\left\Vert A-\varphi_{\lambda}\left(  A\right)  \right\Vert _{\tau}=\left\Vert
PA-P\varphi_{\lambda}\left(  A\right)  \right\Vert _{\tau}\leq\left\Vert
PA-\eta\left(  A\right)  \right\Vert _{\tau}+\left\Vert \eta\left(  A\right)
-E_{\mathcal{E}_{\lambda}}\left(  PA\right)  \right\Vert _{\tau}=
\]%
\[
\left\Vert PA-\eta\left(  A\right)  \right\Vert _{\tau}+\left\Vert
E_{\mathcal{E}_{\lambda}}\left(  \eta\left(  A\right)  \right)
-E_{\mathcal{E}_{\lambda}}\left(  PA\right)  \right\Vert _{\tau}%
\leq2\left\Vert PA-\eta\left(  A\right)  \right\Vert _{\tau}\leq\frac
{2}{k_{\lambda}}.
\]
Clearly,
\[
\lim_{\lambda}\emph{subrank}\left(  \mathcal{B}_{\lambda}\right)  =\infty,
\]
and, since there are sufficiently many tracial states on $\mathcal{M}$, we
have, for every $A\in\mathcal{A}$,%
\[
\varphi_{\lambda}\left(  a\right)  \rightarrow A
\]
in the ultrastrong topology on $\mathcal{M}$. By assumption $\mathcal{A}$ has
no finite-dimensional representations, so it follows from $\left(  8\right)  $
that $\mathcal{A}$ is weakly approximately divisible.

$\left(  10\right)  .$ This follows immediately from $\left(  9\right)  $
since the nuclearity of $\mathcal{A}$ is equivalent to the hyperfiniteness of
$\pi\left(  \mathcal{A}\right)  ^{\prime\prime}$ for every representation
$\pi$ of $\mathcal{A}$.
\end{proof}

\bigskip

\section{Similarity degree\bigskip}

\begin{theorem}
If $\mathcal{A}$ is weakly approximately divisible, then the similarity degree
of $\mathcal{A}$ is at most $5$.
\end{theorem}

\begin{proof}
Suppose $H$ is a Hilbert space and $\rho:\mathcal{A}\rightarrow B\left(
H\right)  $ is a bounded unital homomorphism. Then $\rho$ extends uniquely to
a normal homomorphism $\bar{\rho}:\mathcal{A}^{\#\#}\rightarrow B\left(
H\right)  $. Suppose $A=\left(  a_{ij}\right)  \in\mathcal{M}_{n}\left(
\mathcal{A}\right)  $. Since $\mathcal{A}$ is weakly approximately divisible,
we can choose a net $\left\{  \left(  \mathcal{B}_{\lambda},\varphi_{\lambda
}\right)  \right\}  _{\lambda\in\Lambda}$ as in Definition \ref{WAD}
corresponding to $\mathcal{F}=\left\{  a_{ij}:1\leq i,j\leq n\right\}  $. We
know that%
\[
\bar{\rho}_{n}\left(  \varphi_{\lambda}\left(  a_{ij}\right)  \right)
=\left(  \bar{\rho}\left(  \varphi_{\lambda}\left(  a_{ij}\right)  \right)
\right)  \rightarrow\left(  \bar{\rho}\left(  a_{ij}\right)  \right)
=\rho_{n}\left(  A\right)  ,
\]
where the convergence is in the weak* topology. Moreover, since $\varphi
_{\lambda}$ is completely contractive, we have that
\[
\left\Vert \left(  \varphi_{\lambda}\left(  a_{ij}\right)  \right)
\right\Vert \leq\left\Vert A\right\Vert ,
\]
so
\[
\lim_{\lambda}\left\Vert \left(  \varphi_{\lambda}\left(  a_{ij}\right)
\right)  \right\Vert =\left\Vert A\right\Vert ,
\]
and%
\[
\left\Vert \rho_{n}\left(  A\right)  \right\Vert \leq\limsup_{\lambda
}\left\Vert \bar{\rho}_{n}\left(  \varphi_{\lambda}\left(  a_{ij}\right)
\right)  \right\Vert .
\]
However, $\varphi_{\lambda}\left(  a_{ij}\right)  \in\mathcal{B}_{\lambda
}^{\prime}$ for $1\leq i,j\leq n$ and $\lim_{\lambda}\emph{subrank}\left(
\mathcal{B}_{\lambda}\right)  =\infty$. So the remainder of the proof follows
from Lemma 3.1 in \cite{li}.
\end{proof}

\bigskip

In \cite{F. Pop} F. Pop proved that if $\mathcal{A}$ is a nuclear C*-algebra
containing copies of $\mathcal{M}_{n}\left(  \mathbb{C}\right)  $ for
arbitrarily large values of $n$, then the similarity degree of $\mathcal{A}%
\otimes\mathcal{B}$ is at most $5$ for every unital C*-algebra $\mathcal{B}$.
In \cite{li} the second author showed that this result remains true if
$\mathcal{A}$ is nuclear and contains finite-dimensional algebras with
arbitrarily large subrank. It was shown by \cite{JW} that if $\mathcal{A}$ is
nuclear and contains homomorphic images of certain dimension drop C*-algebras
$\mathcal{Z}_{p,q}$ for all relatively prime integers $p,q$ (e.g.,
$\mathcal{A}$ contains a copy of the Jiang-Su algebra), then, for every unital
C*-algebra $\mathcal{B}$, the similarity degree of $\mathcal{A}\otimes
\mathcal{B}$ is at most $5$. The following corollary includes all of these results.

\bigskip

\begin{corollary}
If $\mathcal{A}$ is a unital tracially nuclear NFD C*-algebra, then, for every
unital C*-algebra $\mathcal{B}$, the similarity degree of $\mathcal{A}%
\otimes\mathcal{B}$ is at most $5$.
\end{corollary}


\bigskip

\begin{thebibliography}{99}                                                                                               %


\bibitem {b}B. Blackadar, A. Kumjian and M. R$\phi$rdam, \emph{Approximately
central matrix units and the structure of noncommutative tori}, K-theory,
6(1992), 267-284.

\bibitem {Bunce1}J. W. Bunce, \emph{The similarity problem for representations
of C$^{*}$-algebras}, Proc. Amer. Math. Soc. 81(1981), 409-414.

\bibitem {Christensen1}E. Christensen, \emph{Extensions of derivations II},
Math. Scand. 50(1982), 111-122.

\bibitem {Ch2}E. Christensen, \emph{On nonselfadjoint representations of
C*-algebras}, Amer. J. Math. 103 (1981) 817--833.

\bibitem {Christensen2}E. Christensen, \emph{Finite von Neumann algebra
factors with property $\Gamma$}, J. Funct. Anal. 186(2001), 366-380.

\bibitem {Dix}J. Dixmier, \emph{\'{E}tude sur les vari\'{e}t\'{e}s et les
op\'{e}rateurs de Julia, avec quelques applications}, Bull. Soc. Math. France
77(1949), 11--101.

\bibitem {F}C. Foia\c{s}, \emph{Invariant para-closed subspaces}, Indiana
Univ. Math. J. 21 (1971/72) 887-906.

\bibitem {Haagerup 1}U. Haagerup, \emph{Solution of the similarity problem for
cyclic representations of C$^{\ast}$-algebras}, Ann. Math. 118(1983), 215-240.

\bibitem {H}D. Hadwin, \emph{Dilations and Hahn decompositions for linear
maps}, Canad. J. Math. 33 (1981) 826--839.

\bibitem {HP}D. Hadwin and V. Paulsen, \emph{Two reformulations of Kadison's
similarity problem}. J. Operator Theory 55 (2006) 3--16.

\bibitem {JW}M. Johanesov\'{a} and W. Winter, \emph{The similarity problem for
}$\mathcal{Z}$\emph{-stable C*-algebras}, Preprint 1104.2067, 2011.

\bibitem {Kadison1}R. Kadison, \emph{On the orthogonalization of operator
representations}, Amer. J. Math. 77(1955), 600-622.

\bibitem {Kir}E. Kirchberg, \emph{The derivation problem and the similarity
problem are equivalent}, J. Operator Theory 36 (1996), no. 1, 59--62

\bibitem {li}W. Li, \emph{The similarity degree of approximately divisible
C*-algebras}, Preprint, 2012, to appear, Operators and Matrices.

\bibitem {li-shen}W. Li and J. Shen, \emph{ A note on approximately divisible
C*-algebras}, Preprint arXiv 0804.0465.

\bibitem {Paulsen}V. I. Paulsen, \emph{Completely bounded maps and dilations}.
Pitman Research Notes in Mathematics Series, 146. Longman Scientific \&
Technical, Harlow; John Wiley \& Sons, Inc., New York, 1986.

\bibitem {P2}V. I Paulsen, \emph{Completely bounded maps on C*-algebras and
invariant operator ranges}, Proc. Amer. Math. Soc. 86 (1982) 91-96.

\bibitem {Pisier 1}G. Pisier, \emph{The similarity degree of an operator
algebra}, Algebra i Analiz 10(1998), 132--

186; translation in St. Petersburg Math. J. 10(1999), 103--146.

\bibitem {Pisier4}G. Pisier, \emph{Remarks on the similarity degree of an
operator algebra}, Internat. J. Math. 12(2001), 403-414.

\bibitem {Pisier3}G. Pisier, \emph{Similarity problems and length}, Taiwanese
J. Math. 5(2001), 1-17.

\bibitem {Pisier 5}G. Pisier, \emph{Similarity problems and completely bounded
maps}, Second, expanded edition, Lecture Notes in Math., vol. 1618,
Springer-Verlag, Berlin 2001.

\bibitem {Pisier2}G. Pisier, \emph{A similarity degree characterization of
nuclear C$^{*}$-algebras}, Publ. Res. Inst. Math. Sci. 42(2006), No. 3, 691-704.

\bibitem {F. Pop}F. Pop, \emph{The similarity problem for tensor products of
certain C$^{\ast}$-algebras}, Bull. Austral. Math. Soc. 70(2004), 385-389.

\bibitem {Tak}M. Takesaki, \emph{Theory of operator algebras. I},
Springer-Verlag, New York-Heidelberg, 1979.

\bibitem {Tak 2}M. Takesaki, (2003), \emph{Nuclear C*-algebras}, Theory of
operator algebras. III, Encyclopaedia of Mathematical Sciences, 127, Berlin,
New York: Springer-Verlag, pp. 153--204.

\bibitem {W}G. Wittstock, \emph{Ein operatorwertiger Hahn-Banach Satz}, J.
Funct. Anal. 40 (1981) 127--150.
\end{thebibliography}
\end{document}